\documentclass[11pt,twoside,leqno]{article}
\usepackage[T1]{fontenc}
\usepackage{amsmath,amssymb,amsthm,mathrsfs}
\usepackage{txfonts}
\usepackage{microtype}
\usepackage[colorlinks=true,linkcolor=blue,citecolor=blue,urlcolor=blue,hyperfootnotes=false]{hyperref}
\allowdisplaybreaks
\numberwithin{equation}{section}
\newtheorem{theorem}{Theorem}[section]
\newtheorem{lemma}[theorem]{Lemma}
\newtheorem{proposition}[theorem]{Proposition}
\newtheorem{corollary}[theorem]{Corollary}
\theoremstyle{remark}
\newtheorem{remark}[theorem]{Remark}
\newcommand{\R}{\mathbb R}
\newcommand{\cR}{\mathscr R}
\newcommand{\Dom}{\operatorname{Dom}}
\newcommand{\Ran}{\operatorname{Ran}}

\newcommand{\ind}{\mathbf 1}
\title{\bf The Coulhon--Duong conjecture for the Riesz transform\\
on complete Riemannian manifolds}
\author{Rui Chen, Renjin Jiang, Bo Li \& Hong-Quan Li}
\date{}
\begin{document}
\maketitle
\begingroup
\renewcommand{\thefootnote}{}
\footnotetext{\hspace{-1.8em}\textit{2020 Mathematics Subject Classification.}
42B20, 58J35, 47B90.\newline
\textit{Key words and phrases.} Riesz transform, complete Riemannian manifold,
weak type estimate, obstacle problem, sub-Markovian semigroup.}
\endgroup
\begin{center}
\begin{minipage}{12.5cm}\small
\textbf{Abstract.} Let $M$ be a complete, non-compact Riemannian manifold. We prove that
its Riesz transform is of weak type $(1,1)$, with constant $2$ for real-valued
functions. Consequently, it is bounded on $L^p(M)$ for $1<p\leq2$, with
constants depending only on $p$, which proves the Coulhon--Duong conjecture.
The proof uses an obstacle decomposition for positive self-adjoint operators
with sub-Markovian semigroups.
Applying this decomposition to the shifted square-root Laplacian and using locality of the Sobolev differential
yields the endpoint estimate without geometric or heat kernel assumptions.
 We also obtain the
corresponding result for Dirichlet spaces admitting a local Hilbertian
differential calculus.
\end{minipage}
\end{center}

\section{Introduction and main results}\label{sec:introduction}
Let $(M,g)$ be a complete, non-compact Riemannian manifold without boundary, and let
$\mu$ be its Riemannian measure. Denote by $\Delta$ the nonnegative
Laplace--Beltrami operator. In this paper, we study the Riesz transform
\[
\cR=d\Delta^{-1/2},
\]
where $d$ is the exterior differential on functions. By the metric
identification of $T^*M$ and $TM$, this is also the operator
$\nabla\Delta^{-1/2}$. The transform is defined on $L^2(M)$ by polar
decomposition and is taken to be zero on $\ker\Delta$.

A conjecture of Coulhon and Duong \cite{CoulhonDuong2003} asserts that
$\cR$ is bounded on $L^p(M)$ for every $1<p\leq2$, on any complete
Riemannian manifold. In particular, the conjecture predicts
\begin{equation}\label{eq:CD-conj}
 \|\nabla\varphi\|_{L^p(M)}
 \leq C_p\|\Delta^{1/2}\varphi\|_{L^p(M)},
 \qquad \varphi\in C_c^\infty(M),\quad 1<p\leq2.
\end{equation}
The case $p=2$ follows from the quadratic-form identity
$\|d\varphi\|_2=\|\Delta^{1/2}\varphi\|_2$. For $1<p<2$, it suffices
to establish a weak type $(1,1)$ estimate and then apply interpolation.
Our main result is the following.

\begin{theorem}\label{thm:main-weak}
Let $(M,g)$ be a complete Riemannian manifold. For every real-valued
$f\in L^1(M)\cap L^2(M)$ and every $\lambda>0$,
\begin{equation}\label{eq:main-weak}
 \mu\bigl(\{x\in M:|\cR f(x)|>\lambda\}\bigr)
 \leq \frac{2}{\lambda}\|f\|_{L^1(M)}.
\end{equation}
Moreover, $\cR$ has a unique continuous linear extension from
$L^1(M;\R)$ to $L^{1,\infty}(M;T^*M)$ satisfying the same estimate.
\end{theorem}

\begin{corollary}\label{cor:CD}
For every $1<p\leq2$, the Riesz transform extends to a bounded operator
from $L^p(M)$ to $L^p(M;T^*M)$, and
\begin{equation}\label{eq:Rp-main}
 \|\cR f\|_{L^p(M;T^*M)}\leq C_p\|f\|_{L^p(M)}.
\end{equation}
Here $C_p$ depends only on $p$, and one may take $C_2=1$.
In particular, \eqref{eq:CD-conj} holds.
\end{corollary}

The estimate in Theorem~\ref{thm:main-weak} is independent of the dimension
and of the geometry of $M$. We impose no volume doubling condition,
curvature bound, heat kernel estimate, or Poincar\'e inequality.
The stated endpoint constant concerns real-valued functions; the
$L^p$ conclusion also holds for complex-valued functions.

The study of the Laplacian and Riesz transforms in this setting goes
back to Strichartz \cite{Strichartz1983}. Bakry
\cite{Bakry1985,Bakry1987} obtained estimates under curvature assumptions.
Around the same time, Chen \cite{Chen1987} established the weak type $(1,1)$
estimate under non-negativity of Ricci curvature, while later Li \cite{Li1991} provided
another proof based on heat kernel regularity.
Coulhon and Duong \cite{CoulhonDuong1999} proved the weak type $(1,1)$
estimate under volume doubling and an on-diagonal heat kernel upper
bound. Their argument requires no pointwise estimate for the gradient
of the heat kernel. On arbitrary complete manifolds, Coulhon and Duong
\cite{CoulhonDuong2003} established the multiplicative inequality
\[
 \|\nabla f\|_p^2\leq C_p\|f\|_p\|\Delta f\|_p,
 \qquad 1<p\leq2,
\]
while Coulhon, Duong and Li \cite{CoulhonDuongLi2003} proved related
Littlewood--Paley--Stein estimates.

Several subsequent results treated geometries beyond the original
Gaussian setting. Chen, Coulhon, Feneuil and Russ
\cite{ChenCoulhonFeneuilRuss2017} studied manifolds with sub-Gaussian
heat kernel estimates, see also Li and Zhu \cite{LiZhu2018}. Hassell and Sikora \cite{HassellSikora2019},
Hassell, Nix and Sikora \cite{HassellNixSikora2024}, and Jiang, Li and Lin
\cite{JiangLiLin2025} obtained results on non-doubling manifolds with ends.
Ouhabaz \cite{Ouhabaz2024} proved a multiplicative estimate involving
$\Delta^{1/2+\varepsilon}$ and $\Delta^{1/2-\varepsilon}$ on general
complete manifolds, approaching the differential order in
\eqref{eq:CD-conj}. For $p>2$, boundedness is more sensitive to geometry;
see Auscher, Coulhon, Duong and Hofmann
\cite{AuscherCoulhonDuongHofmann2004}, Carron, Coulhon and Hassell
\cite{CarronCoulhonHassell2006} and Coulhon, Jiang, Koskela and Sikora \cite{CoulhonJiangKoskelaSikora2020}.
For further results concerning Poincar\'e
inequalities, perturbations, curvature decay, graphs, and domains, we
refer to \cite{AuscherCoulhon2005,BernicotFrey2016,Carron2007,Carron2017, Chen2015,ChenCoulhonHua2020,CoulhonSikora2010,
Devyver2014,Devyver2015,Jiang2021,JiangLin2024,Li1999,Li2010}
and the references therein.

The usual endpoint proof relies on a Calder\'on--Zygmund decomposition
and estimates away from the balls supporting the bad functions.
Without volume doubling or heat kernel bounds, this approach does not
provide the desired estimate. A different decomposition was introduced
by Ouyang, Spector and Stockdale \cite{OuyangSpectorStockdale2026}, who
used a fractional obstacle problem to obtain a dimension-free weak
type $(1,1)$ estimate for the vector Riesz transform on $\R^n$.
Mao, Wang and Zhang \cite{MaoWangZhang2026} adapted this approach to
stratified Lie groups.

Inspired in particular by \cite{OuyangSpectorStockdale2026}, 
we shall use an abstract obstacle decomposition which depends only on
self-adjointness, a positive spectral lower bound, and the
sub-Markov property. It is convenient to state this result separately.

\begin{theorem}\label{thm:abstract-obstacle}
Let $(X,\nu)$ be a $\sigma$-finite measure space, and let $A$ be a
self-adjoint operator on $L^2(X,\nu;\R)$. Suppose that $A\geq aI$ for
some $a>0$ and that $e^{-tA}$ is sub-Markovian. Given a nonnegative
$f\in L^1(X,\nu)\cap L^2(X,\nu)$ and $\lambda>0$, there exist
\[
 u\in\Dom(A)\cap L^1(X,\nu),\qquad
 m\in L^1(X,\nu)\cap L^\infty(X,\nu)
\]
such that
\begin{equation}\label{eq:obstacle}
 u\geq0,\qquad f=Au+m,\qquad 0\leq m\leq\lambda,
 \qquad m=\lambda\ \text{a.e. on }\{u>0\}.
\end{equation}
Furthermore,
\begin{equation}\label{eq:obstacle-bounds}
 \|m\|_1\leq\|f\|_1,
 \qquad \lambda\nu(\{u>0\})\leq\|f\|_1.
\end{equation}
\end{theorem}

For the application to manifolds, take
$A=A_\kappa:=\Delta^{1/2}+\kappa I$, $\kappa>0$.
Applying Theorem~\ref{thm:abstract-obstacle} to $f^+$ and $f^-$ gives
\[
 f=g+A_\kappa v,\qquad
 |g|\leq\lambda,\qquad \|g\|_1\leq\|f\|_1,
\]
where $dv$ vanishes outside a measurable set $\Omega$ with
$\mu(\Omega)\leq\lambda^{-1}\|f\|_1$. Since
$\cR_\kappa:=dA_\kappa^{-1}$ is an $L^2$ contraction,
\[
 \mu(\{|\cR_\kappa f|>\lambda\})
 \leq\mu(\Omega)+\lambda^{-2}\|\cR_\kappa g\|_2^2
 \leq\frac{2}{\lambda}\|f\|_1.
\]
The constant is uniform in $\kappa$. Spectral convergence as
$\kappa\downarrow0$ then proves Theorem~\ref{thm:main-weak}.

The proof of Theorem~\ref{thm:abstract-obstacle} uses the regularized
equation
\[
 Au_\varepsilon+\beta_\varepsilon(u_\varepsilon)=f,
 \qquad
 \beta_\varepsilon(s)=\min\{s^+/\varepsilon,\lambda\}.
\]
This is the Yosida regularization of the contact relation associated
with $\lambda s_+$; see \cite{Brezis1973}. Semigroup truncation
inequalities give positivity, monotonicity in $\varepsilon$, and a
uniform $L^1$ bound for $\beta_\varepsilon(u_\varepsilon)$.
These estimates permit passage to the limit and recover the contact
condition in \eqref{eq:obstacle}. The argument requires neither a
kernel representation nor geometric assumptions on $X$.

The paper is organized as follows. Section~\ref{sec:preliminaries}
collects the operator and truncation facts used in the proof.
Section~\ref{sec:obstacle} proves Theorem~\ref{thm:abstract-obstacle}.
In Section~\ref{sec:riesz}, we prove the uniform shifted estimate,
remove the shift, and deduce Corollary~\ref{cor:CD}.
Section~\ref{sec:extensions} gives the extension to spaces with a local
Hilbertian differential calculus.

Throughout the paper, $\|\cdot\|_p$ denotes the norm in the relevant
$L^p$ space. Functions in the obstacle argument are real-valued, and
all pointwise assertions are understood almost everywhere. We write
$u^+=\max\{u,0\}$ and $u^-=\max\{-u,0\}$.

\section{Preliminaries}\label{sec:preliminaries}
\subsection{The Laplacian and its semigroups}\label{riesz}
On a complete Riemannian manifold, the nonnegative Laplace--Beltrami
operator on $C_c^\infty(M)$ is essentially self-adjoint
\cite{Strichartz1983}. We use the same notation $\Delta$ for its
self-adjoint closure. The differential is a closed densely defined
operator
\[
 d:W^{1,2}(M)\subset L^2(M)\longrightarrow L^2(M;T^*M),
\]
and completeness gives $W^{1,2}(M)=W^{1,2}_0(M)$; see
\cite{Hebey1996,Grigoryan2009}. With the sign convention
$\Delta=d^*d=-\operatorname{div}\nabla$, the associated closed form is
\[
 \mathfrak q(u,v)=\int_M\langle du,dv\rangle\,d\mu,
 \qquad u,v\in W^{1,2}(M).
\]
The representation theorem for closed forms \cite{Kato1995} yields
\begin{equation}\label{eq:form}
 \Dom(\Delta^{1/2})=W^{1,2}(M),\qquad
 \|\Delta^{1/2}u\|_2=\|du\|_2.
\end{equation}
By polar decomposition,
\begin{equation}\label{eq:polar}
 d=\cR\Delta^{1/2},\qquad
 \|\cR\|_{L^2(M)\to L^2(M;T^*M)}\leq1,
\end{equation}
where $\cR$ is isometric on $(\ker\Delta)^\perp$ and zero on
$\ker\Delta$. This defines the $L^2$ Riesz transform, including when
$0$ belongs to the spectrum. It agrees with $d\Delta^{-1/2}$ on
$\Ran\Delta^{1/2}$.

The heat semigroup $H_t=e^{-t\Delta}$ is symmetric and sub-Markovian.
In particular, it preserves positivity and extends consistently to a
contraction on $L^q(M)$ for $1\leq q\leq\infty$; see
\cite{Davies1989,Ouhabaz2005}. Bochner's subordination formula gives
\begin{equation}\label{eq:subordination}
 P_t:=e^{-t\Delta^{1/2}}
 =\frac{t}{2\sqrt\pi}\int_0^\infty
 e^{-t^2/(4s)}H_s\,\frac{ds}{s^{3/2}},\qquad t>0.
\end{equation}
The measure in this formula has total mass one, so $P_t$ is again
sub-Markovian; see \cite{Stein1970}. Thus, for every $\kappa>0$,
\begin{equation}\label{eq:shift}
 A_\kappa=\Delta^{1/2}+\kappa I\geq\kappa I,
 \qquad e^{-tA_\kappa}=e^{-\kappa t}P_t
\end{equation}
satisfy the hypotheses of Theorem~\ref{thm:abstract-obstacle}.
This argument does not require stochastic completeness.

\subsection{Truncation inequalities}
Let $(X,\nu)$ be a $\sigma$-finite measure space. We first record two
inequalities for the generator of a symmetric semigroup. All pairings
in this subsection are real $L^2(X,\nu)$ pairings.

\begin{lemma}\label{lem:positive-coercivity}
Suppose that $A$ is self-adjoint, $A\geq aI$ for some $a>0$, and
$T_t=e^{-tA}$ preserves positivity. Then, for every $w\in\Dom(A)$,
\[
 \langle Aw,w^+\rangle\geq a\|w^+\|_2^2,
 \qquad
 \langle Aw,w^-\rangle\leq-a\|w^-\|_2^2.
\]
\end{lemma}
\begin{proof}
Since $w=w^+-w^-$ and $w^+w^-=0$, positivity gives
\[
 \langle w-S_tw,w^+\rangle
 =\|w^+\|_2^2-\langle T_tw^+,w^+\rangle
   +\langle T_tw^-,w^+\rangle
 \geq(1-e^{-at})\|w^+\|_2^2.
\]
Here we used $\|T_t\|_{2\to2}\leq e^{-at}$, which follows from the
spectral theorem. Divide by $t$ and let $t\downarrow0$.
As $t^{-1}(I-T_t)w\to Aw$ in $L^2$, the first inequality follows.
Applying it to $-w$ proves the second.
\end{proof}

For $\delta>0$, set
\[
 \tau_\delta(s)=\min\{s^+/\delta,1\},\qquad
 \Phi_\delta(s)=\int_0^s\tau_\delta(r)\,dr.
\]
Then $\Phi_\delta$ is nonnegative and convex, $\Phi_\delta(0)=0$,
and
\begin{equation}\label{eq:truncation-size}
 |\tau_\delta(s)|\leq |s|/\delta,
 \qquad 0\leq\Phi_\delta(s)\leq s^2/(2\delta).
\end{equation}
In particular, $\tau_\delta(w)\in L^2$ and $\Phi_\delta(w)\in L^1$
whenever $w\in L^2$.

\begin{lemma}\label{lem:truncation-dissipation}
Let $A$ be a nonnegative self-adjoint operator such that $T_t=e^{-tA}$ is
sub-Markovian. For $w\in\Dom(A)$ and $\delta>0$,
\[
 \langle Aw,\tau_\delta(w)\rangle\geq0.
\]
\end{lemma}
\begin{proof}
The sub-Markov property gives the Jensen inequality
\begin{equation}\label{eq:jensen}
 \Phi_\delta(T_tv)\leq T_t\Phi_\delta(v).
\end{equation}
For completeness, if $v$ is bounded and $\alpha s+\beta$ is a supporting
affine function for $\Phi_\delta$, then $\beta\leq0$. Positivity and
$T_t1\leq1$ imply
\[
 T_t\Phi_\delta(v)\geq\alpha T_tv+\beta T_t1
 \geq\alpha T_tv+\beta.
\]
Taking the supremum over a countable family of supporting affine
functions proves \eqref{eq:jensen}. For $v\in L^2$, apply this to the
bounded truncations of $v$. They converge in $L^2$, and their images
under $\Phi_\delta$ converge in $L^1$ by
\eqref{eq:truncation-size} and dominated convergence. Contractivity of
$T_t$ on these spaces and passage to an almost everywhere convergent
subsequence give \eqref{eq:jensen} for $v\in L^2$.

Since $T_t$ is an $L^1$ contraction, \eqref{eq:jensen} yields
$\int_X\Phi_\delta(T_tw)\,d\nu\leq\int_X\Phi_\delta(w)\,d\nu$.
Convexity also gives
\[
 \Phi_\delta(w)-\Phi_\delta(T_tw)
 \leq\tau_\delta(w)(w-T_tw).
\]
After integration, we obtain
$\langle w-T_tw,\tau_\delta(w)\rangle\geq0$.
Divide by $t$ and let $t\downarrow0$ to finish the proof.
\end{proof}

\section{An abstract obstacle decomposition}\label{sec:obstacle}
In this section, we prove Theorem~\ref{thm:abstract-obstacle}. Fix
$f\geq0$ in $L^1(X,\nu)\cap L^2(X,\nu)$ and $\lambda>0$.
For $\varepsilon>0$, define
\begin{equation}\label{eq:penalty}
 \beta_\varepsilon(s)=\min\{s^+/\varepsilon,\lambda\},
 \qquad \alpha_\varepsilon(s)=\int_0^s\beta_\varepsilon(r)\,dr.
\end{equation}
The function $\alpha_\varepsilon$ is convex and continuously differentiable,
and $0\leq \alpha_\varepsilon(s)\leq s^2/(2\varepsilon)$.
We shall first solve the penalized equation and then pass to the limit.

\begin{proof}[Proof of Theorem~\ref{thm:abstract-obstacle}]
\textbf{Step 1. The penalized equation.}
On $\Dom(A^{1/2})$, consider
\[
 J_\varepsilon(\phi)=\frac12\|A^{1/2}\phi\|_2^2
 +\int_X \alpha_\varepsilon(\phi)\,d\nu-\langle f,\phi\rangle.
\]
Since $A\geq aI$,
\[
 J_\varepsilon(\phi)\geq\frac14\|A^{1/2}\phi\|_2^2-a^{-1}\|f\|_2^2.
\]
Thus $J_\varepsilon$ is coercive and strictly convex. The integral
term is continuous and convex on $L^2$: indeed, the Lipschitz bound
for $\beta_\varepsilon$ gives
\[
 \left|\alpha_\varepsilon(s)-\alpha_\varepsilon(t)
       -\beta_\varepsilon(t)(s-t)\right|
 \leq\frac{|s-t|^2}{2\varepsilon}.
\]
It follows that $J_\varepsilon$ is weakly lower semicontinuous on the
form domain. The direct method gives a unique minimizer
$u_\varepsilon\in\Dom(A^{1/2})$.

The Euler equation reads
\[
 \langle A^{1/2}u_\varepsilon,A^{1/2}\phi\rangle
 +\langle\beta_\varepsilon(u_\varepsilon),\phi\rangle
 =\langle f,\phi\rangle,\qquad \phi\in\Dom(A^{1/2}).
\]
As $|\beta_\varepsilon(s)|\leq|s|/\varepsilon$, the penalization term
belongs to $L^2$. The representation theorem for closed forms therefore
implies $u_\varepsilon\in\Dom(A)$ and
\begin{equation}\label{eq:penalized-equation}
 Au_\varepsilon+m_\varepsilon=f,
 \qquad m_\varepsilon=\beta_\varepsilon(u_\varepsilon).
\end{equation}

\smallskip
\noindent\textbf{Step 2. Positivity and uniform estimates.}
Pair \eqref{eq:penalized-equation} with $u_\varepsilon^-$.
Since $m_\varepsilon u_\varepsilon^-=0$, Lemma~\ref{lem:positive-coercivity}
gives
\[
 0\leq\langle f,u_\varepsilon^-\rangle
 =\langle Au_\varepsilon,u_\varepsilon^-\rangle
 \leq-a\|u_\varepsilon^-\|_2^2.
\]
Consequently, $u_\varepsilon\geq0$ and $0\leq m_\varepsilon\leq\lambda$.
Pairing instead with $\tau_\delta(u_\varepsilon)$ and using
Lemma~\ref{lem:truncation-dissipation}, we obtain
\[
 \int_Xm_\varepsilon\tau_\delta(u_\varepsilon)\,d\nu
 \leq\int_Xf\tau_\delta(u_\varepsilon)\,d\nu\leq\|f\|_1.
\]
Now $\tau_\delta(u_\varepsilon)\uparrow\ind_{\{u_\varepsilon>0\}}$
as $\delta\downarrow0$, and $m_\varepsilon=0$ on
$\{u_\varepsilon=0\}$. Monotone convergence gives
\begin{equation}\label{eq:penalized-bounds}
 \|m_\varepsilon\|_1\leq\|f\|_1,
 \qquad \|m_\varepsilon\|_2^2\leq\lambda\|f\|_1.
\end{equation}

\smallskip
\noindent\textbf{Step 3. Monotonicity and convergence.}
If $0<\varepsilon\leq\eta$, then
$\beta_\varepsilon(s)\geq\beta_\eta(s)$ for $s\geq0$.
Subtract the two penalized equations and pair with
$w=(u_\varepsilon-u_\eta)^+$. On $\{w>0\}$,
\[
 \beta_\varepsilon(u_\varepsilon)
 \geq\beta_\eta(u_\varepsilon)\geq\beta_\eta(u_\eta).
\]
Lemma~\ref{lem:positive-coercivity} therefore implies
\[
 0=\langle A(u_\varepsilon-u_\eta),w\rangle
 +\langle\beta_\varepsilon(u_\varepsilon)
            -\beta_\eta(u_\eta),w\rangle
 \geq a\|w\|_2^2.
\]
Hence $0\leq u_\varepsilon\leq u_\eta$.

Choose $\varepsilon_n\downarrow0$. The sequence $u_{\varepsilon_n}$
decreases almost everywhere to a nonnegative function $u$, and is
dominated by $u_{\varepsilon_1}\in L^2$. Thus
\begin{equation}\label{eq:convergence}
 u_{\varepsilon_n}\longrightarrow u\quad\text{in }L^2,
 \qquad m_{\varepsilon_n}\rightharpoonup m\quad\text{in }L^2,
\end{equation}
where the second convergence follows from \eqref{eq:penalized-bounds}, after passing to a subsequence.
The set 
$\{h\in L^2:0\leq h\leq\lambda\}$ is  weakly
closed, and therefore $0\leq m\leq\lambda$.

\smallskip
\noindent\textbf{Step 4. The limit equation and the $L^1$ bound.}
For every measurable $E\subset X$ of finite measure,
\[
 \int_E m\,d\nu=\lim_{n\to\infty}\int_E m_{\varepsilon_n}\,d\nu
 \leq\|f\|_1.
\]
Exhausting $X$ by sets of finite measure gives
$\|m\|_1\leq\|f\|_1$. For $\phi\in\Dom(A)$,
self-adjointness and \eqref{eq:penalized-equation} give
\[
 \langle u_{\varepsilon_n},A\phi\rangle
 =\langle f-m_{\varepsilon_n},\phi\rangle.
\]
Pass to the limit using \eqref{eq:convergence}. We obtain
$\langle u,A\phi\rangle=\langle f-m,\phi\rangle$ for every $\phi\in\Dom(A)$.
Since $A=A^*$, this proves $u\in\Dom(A)$ and $Au=f-m$.

\smallskip
\noindent\textbf{Step 5. The contact relation.}
For $k\geq1$, put $E_k=\{u>1/k\}$. Then
$\nu(E_k)\leq k^2\|u\|_2^2<\infty$. For all sufficiently large $j$,
$\varepsilon_j\lambda<1/k$, and the monotonicity from Step~3 gives
\[
 u_{\varepsilon_j}\geq u>1/k>\varepsilon_j\lambda
 \quad\text{on }E_k.
\]
Thus $m_{\varepsilon_j}=\lambda$ on $E_k$. Testing the weak convergence
in \eqref{eq:convergence} against functions supported in $E_k$ gives
$m=\lambda$ there. Since $\{u>0\}=\bigcup_{k\geq1}E_k$, the contact
relation follows. Finally,
\[
 \lambda\nu(\{u>0\})\leq\|m\|_1\leq\|f\|_1,
 \qquad
 \|u\|_1\leq\nu(\{u>0\})^{1/2}\|u\|_2<\infty.
\]
This proves all the assertions.
\end{proof}

\begin{remark}
The contact relation is equivalently
$m\in\partial\Phi(u)$, where
$\Phi(\phi)=\lambda\int_X\phi_+\,d\nu$ is regarded as an extended-valued
convex functional on $L^2(X,\nu)$. The proof above realizes its
Moreau--Yosida regularization explicitly. The positive spectral lower
bound provides coercivity and the comparison in Step~3; the
sub-Markov property supplies the $L^1$ estimate in Step~2.
\end{remark}

\section{The weak type estimate and its consequences}\label{sec:riesz}
\subsection{Uniform estimates for the shifted transforms}\label{sec:endpoint}
For $\kappa>0$, let $A_\kappa$ be as in \eqref{eq:shift} and set
$\cR_\kappa=dA_\kappa^{-1}$. By \eqref{eq:form},
$\Dom(A_\kappa)=W^{1,2}(M)$. From \eqref{eq:polar},
\begin{equation}\label{eq:Rk-polar}
 \cR_\kappa=\cR\Delta^{1/2}(\Delta^{1/2}+\kappa I)^{-1}.
\end{equation}
The spectral multiplier $\sqrt\rho/(\sqrt\rho+\kappa)$ takes values
in $[0,1]$. Hence
\begin{equation}\label{eq:Rk-L2}
 \|\cR_\kappa f\|_2\leq\|f\|_2,\qquad f\in L^2(M),
\end{equation}
uniformly in $\kappa$.

We also require locality of the differential on level sets.
\begin{lemma}\label{lem:level-locality}
If $u\in W^{1,2}(M)$ is real-valued and $c\in\R$, then
$\ind_{\{u=c\}}du=0$ almost everywhere.
\end{lemma}
\begin{proof}
Work locally in a coordinate chart, so that $u-c\in W^{1,2}_{\rm loc}$.
The Sobolev truncation formulas give
\[
 d(u-c)^+=\ind_{\{u>c\}}du,
 \qquad d(u-c)^-=-\ind_{\{u<c\}}du.
\]
Subtracting yields $du=\ind_{\{u\ne c\}}du$. A countable covering
by charts proves the assertion on $M$.
\end{proof}

\begin{proposition}\label{prop:shifted-weak}
For every $\kappa>0$, every real-valued $f\in L^1(M)\cap L^2(M)$,
and every $\lambda>0$,
\[
 \mu(\{|\cR_\kappa f|>\lambda\})\leq\frac{2}{\lambda}\|f\|_1.
\]
\end{proposition}
\begin{proof}
By \eqref{eq:shift}, Theorem~\ref{thm:abstract-obstacle} applies to
$A_\kappa$. Apply it to $f^+$ and $f^-$ at the same level $\lambda$:
\[
 f^\pm=A_\kappa u_\pm+m_\pm,\qquad
 u_\pm\geq0,\quad 0\leq m_\pm\leq\lambda.
\]
Set
\[
 g=m_+-m_-,\qquad w=u_+-u_-,\qquad
 \Omega=\{u_+>0\}\cup\{u_->0\}.
\]
Then $f=g+A_\kappa w$.
By \eqref{eq:obstacle} and \eqref{eq:obstacle-bounds}, we deduce that
\begin{equation}\label{eq:good-bounds}
 |g|\leq\lambda,\quad \|g\|_1\leq\|f\|_1,\quad
 \|g\|_2^2\leq\lambda\|f\|_1,\quad
 \mu(\Omega)\leq\lambda^{-1}\|f\|_1.
\end{equation}
Both $u_+$ and $u_-$ belong to $W^{1,2}(M)$ and vanish on
$M\setminus\Omega$. Lemma~\ref{lem:level-locality} implies $dw=0$
there. Applying $dA_\kappa^{-1}$ to the decomposition gives
\[
 \cR_\kappa f=\cR_\kappa g+dw
 =\cR_\kappa g\quad\text{on }M\setminus\Omega.
\]
Chebyshev's inequality, \eqref{eq:Rk-L2}, and \eqref{eq:good-bounds}
now yield
\begin{align*}
 \mu(\{|\cR_\kappa f|>\lambda\})
 &\leq\mu(\Omega)+\mu(\{|\cR_\kappa g|>\lambda\})\\
 &\leq\lambda^{-1}\|f\|_1+\lambda^{-2}\|\cR_\kappa g\|_2^2
 \leq\frac{2}{\lambda}\|f\|_1.
\end{align*}
This completes the proof.
\end{proof}

\subsection{Removal of the shift}\label{sec:limit}
Let $\{E_\rho\}$ be the spectral resolution of $\Delta$. By
\eqref{eq:Rk-polar} and the fact that $\cR$ vanishes on $\ker\Delta$,
\[
 \|\cR_\kappa f-\cR f\|_2^2
 =\int_{(0,\infty)}
 \left(\frac{\kappa}{\sqrt\rho+\kappa}\right)^2
 \,d\langle E_\rho f,f\rangle.
\]
Dominated convergence proves
\begin{equation}\label{eq:Rk-strong}
 \cR_\kappa f\longrightarrow\cR f
 \quad\text{in }L^2(M;T^*M),\qquad f\in L^2(M).
\end{equation}
No spectral gap for $\Delta$ is needed.

\begin{proof}[Proof of Theorem~\ref{thm:main-weak}]
For real-valued $f\in L^1(M)\cap L^2(M)$, choose
$\kappa_j\downarrow0$ so that, after passing to a subsequence in
\eqref{eq:Rk-strong}, $\cR_{\kappa_j}f\to\cR f$ almost everywhere.
For each $\lambda>0$,
\[
 \ind_{\{|\cR f|>\lambda\}}
 \leq\liminf_{j\to\infty}\ind_{\{|\cR_{\kappa_j}f|>\lambda\}}.
\]
Fatou's lemma and Proposition~\ref{prop:shifted-weak} prove
\eqref{eq:main-weak}.

We include the extension argument, since the target is a weak $L^1$
space. Given $f\in L^1(M;\R)$, take real-valued
$f_n\in L^1\cap L^2$ converging to $f$ in $L^1$. The estimate just
proved implies
\begin{equation}\label{eq:cauchy-measure}
 \mu(\{|\cR f_n-\cR f_m|>t\})
 \leq\frac{2}{t}\|f_n-f_m\|_1,\qquad t>0.
\end{equation}
Choose a subsequence with
$\|f_{n_{j+1}}-f_{n_j}\|_1\leq2^{-3j}$. The sets
\[
 E_j=\{|\cR f_{n_{j+1}}-\cR f_{n_j}|>2^{-j}\}
\]
have summable measures. Outside $\limsup_jE_j$, the corresponding
series of increments converges absolutely in $T_x^*M$. Denote the
resulting measurable limit by $F$. Moreover, outside
$\bigcup_{k\geq j}E_k$,
$|\cR f_{n_j}-F|\leq2^{1-j}$, so the subsequence converges to $F$
in measure. Equation~\eqref{eq:cauchy-measure} then gives convergence
of the full sequence in measure.

The limit is independent of the approximating sequence and respects
linear combinations. Define $\cR f=F$. Fatou's lemma along the
almost everywhere convergent subsequence gives
\[
 \mu(\{|\cR f|>\lambda\})
 \leq\liminf_j\mu(\{|\cR f_{n_j}|>\lambda\})
 \leq\frac{2}{\lambda}\|f\|_1.
\]
Thus the extension is continuous for the quasi-norm
\[
 \|\omega\|_{1,\infty}
 =\sup_{\lambda>0}\lambda\mu(\{|\omega|>\lambda\}).
\]
Applying the same estimate to differences proves continuity and,
by density, uniqueness.
\end{proof}

\subsection{The \texorpdfstring{$L^p$}{Lp} estimates}\label{lp}
\begin{proof}[Proof of Corollary~\ref{cor:CD}]
Apply the Marcinkiewicz interpolation theorem to the sublinear map
$f\mapsto|\cR f|$, using Theorem~\ref{thm:main-weak} and
\eqref{eq:polar}. This gives \eqref{eq:Rp-main} for real-valued
$f\in L^p\cap L^2$, $1<p<2$, with $C_p$ depending only on $p$.
Density gives the extension to real $L^p$. Splitting into real and
imaginary parts gives the complex-valued assertion, with at most an
absolute change in the constant. For $p=2$, \eqref{eq:polar} gives
$C_2=1$ directly, also for complex-valued functions.

To deduce \eqref{eq:CD-conj}, we verify that
$\Delta^{1/2}\varphi\in L^p\cap L^2$ for
$\varphi\in C_c^\infty(M)$. The spectral theorem gives
\begin{equation}\label{eq:square-root-semigroup}
 \Delta^{1/2}\varphi
 =\frac{1}{2\sqrt\pi}\int_0^\infty
 (I-H_t)\varphi\,\frac{dt}{t^{3/2}}
\end{equation}
in $L^2$. Contractivity of $H_t$ and
$(I-H_t)\varphi=\int_0^tH_s\Delta\varphi\,ds$ imply
\[
 \|(I-H_t)\varphi\|_p
 \leq\min\{t\|\Delta\varphi\|_p,\,2\|\varphi\|_p\}.
\]
Hence the integral in \eqref{eq:square-root-semigroup} converges
absolutely in $L^p$. Its $L^p$ and $L^2$ limits agree locally in
measure, and therefore almost everywhere. Now \eqref{eq:polar}
and \eqref{eq:Rp-main} give
\[
 \|\nabla\varphi\|_p=\|d\varphi\|_p
 =\|\cR\Delta^{1/2}\varphi\|_p
 \leq C_p\|\Delta^{1/2}\varphi\|_p.
\]
For $p=2$, the asserted identity follows from \eqref{eq:form}.
\end{proof}

\section{Extensions and remarks}\label{sec:extensions}
The argument uses the manifold structure only through the
factorization $\Delta=d^*d$ and locality of $d$. We formulate the
corresponding extension in a setting where the pointwise norm of the
differential is available.

Let $(X,\nu)$ be $\sigma$-finite, and let $\mathcal H$ be a Hilbert
space of square-integrable measurable sections with pointwise norm
$|\cdot|$, so that $\|\omega\|_{\mathcal H}^2=
\int_X|\omega|^2\,d\nu$. Suppose that
\[
 D:\mathcal V\subset L^2(X,\nu;\R)\longrightarrow\mathcal H
\]
is closed and densely defined. Set $L=D^*D$, and assume that
$e^{-tL}$ is sub-Markovian and
\begin{equation}\label{eq:abstract-locality}
 \ind_{\{u=0\}}Du=0,\qquad u\in\mathcal V.
\end{equation}
The same formulation applies to a Hilbert $L^2$-module equipped with
its pointwise norm and multiplication by characteristic functions.
Define $\mathcal T$ by the polar decomposition $D=\mathcal T L^{1/2}$,
with $\mathcal T=0$ on $\ker L$.

\begin{theorem}\label{thm:local-extension}
Under the preceding assumptions, for every real-valued
$f\in L^1(X,\nu)\cap L^2(X,\nu)$ and $\lambda>0$,
\[
 \nu(\{|\mathcal T f|>\lambda\})\leq\frac{2}{\lambda}\|f\|_1.
\]
Consequently, $\||\mathcal T f|\|_p\leq C_p\|f\|_p$ for
$f\in L^p\cap L^2$, $1<p\leq2$, with constants depending only on $p$.
The transform extends continuously to $L^1$ with values in the weak
$L^1$ space of sections, and to $L^p$ with values in the corresponding
$L^p$ space, for $1<p\leq2$.
\end{theorem}
\begin{proof}
The representation theorem gives $\Dom(L^{1/2})=\mathcal V$ and
$\|L^{1/2}u\|_2=\|Du\|_{\mathcal H}$. By subordination,
$L^{1/2}+\kappa I$ satisfies  the hypotheses of Theorem~\ref{thm:abstract-obstacle}.
Apply that theorem to $f^+$ and $f^-$ and use the notation
$g,v,\Omega$ from Proposition~\ref{prop:shifted-weak}.
Equation~\eqref{eq:abstract-locality} gives $Dv=0$ on
$X\setminus\Omega$, while spectral calculus gives
\[
 \|D(L^{1/2}+\kappa I)^{-1}\|_{L^2\to\mathcal H}\leq1.
\]
The proof of Proposition~\ref{prop:shifted-weak} now applies verbatim.
The spectral convergence in \eqref{eq:Rk-strong}, with $L$ in place of
$\Delta$, gives the unshifted estimate.
The extension argument of Section~\ref{sec:limit}, followed by the interpolation and density argument of Section~\ref{lp}, proves the remaining assertions.
For a Hilbert module, the measurable-section realization, or equivalently
its completion in the corresponding pointwise-norm topology, gives
the same extension argument.
\end{proof}

In particular, Theorem~\ref{thm:local-extension} applies to strongly
local Hilbertian Dirichlet spaces equipped with such a differential
calculus. It also applies to infinitesimally Hilbertian metric measure
spaces whenever their differential and heat flow have the stated
properties. These properties hold for the standard calculus on
$\operatorname{RCD}(K,N)$ and $\operatorname{RCD}(K,\infty)$ spaces;
see \cite{AmbrosioGigliSavare2014,Gigli2018}.
Previous Riesz-transform estimates in these settings include
\cite{JiangLiZhang2016,CarbonaroTamaniniTrevisan2023}, and estimates
on tamed Dirichlet spaces were obtained in \cite{EsakiXuKuwae2023}.
The endpoint proof above depends on the explicitly stated semigroup
and locality properties, with no curvature parameter in the constant.

\begin{remark}
The positive shift is used only to construct the obstacle decomposition.
It is removed at the $L^2$ level, where polar decomposition automatically
accounts for $\ker L$. Thus no assumption on the total measure, a
spectral gap for $L$, or preservation of constants by the heat semigroup
is required.
\end{remark}

\subsection*{Acknowledgments}
R. Jiang was partially supported by NNSF of China (12526205 \& 12471094).
B. Li    was partially supported by NNSF of China (12571121), and Zhejiang NSF of China (LMS26A010015).
H. Li was  partially supported by NNSF of China (12671116).
During the preparation of the manuscript, the authors used GPT-6 for assistance with concrete
computations and language editing. The authors take responsibility
for the mathematical arguments and the final content of the paper.

\providecommand{\bysame}{\leavevmode\hbox to3em{\hrulefill}\thinspace}
\providecommand{\MR}{\relax\ifhmode\unskip\space\fi MR }
\providecommand{\MRhref}[2]{%
  \href{http://www.ams.org/mathscinet-getitem?mr=#1}{#2}
}
\providecommand{\href}[2]{#2}

\bigskip
\noindent\textsc{Rui Chen}\par
\noindent School of Mathematical Sciences, Fudan University,
Shanghai 200433, China\par
\noindent\textit{E-mail address:} \texttt{chenrui23@m.fudan.edu.cn}

\medskip
\noindent\textsc{Renjin Jiang}\par
\noindent Academy for Multidisciplinary Studies, Capital Normal University,
Beijing 100048, China\par
\noindent\textit{E-mail address:} \texttt{rejiang@cnu.edu.cn}

\medskip
\noindent\textsc{Bo Li}\par
\noindent Institute of Mathematics, Jiaxing University,
Jiaxing 314001, China\par
\noindent\textit{E-mail address:} \texttt{bli@zjxu.edu.cn}

\bigskip
\noindent\textsc{Hong-Quan Li}\par
\noindent School of Mathematical Sciences, Fudan University,
Shanghai 200433, China\par
\noindent\textit{E-mail address:} \texttt{hongquan\_li@fudan.edu.cn}

\end{document}